\documentclass[11pt]{article}
\makeatletter

\usepackage{amssymb,amsmath,amsthm,latexsym}
\usepackage[mathscr]{eucal} \usepackage{mathrsfs}
\usepackage{comment}
\usepackage{color}
\usepackage{hyperref}
 
\title{Scattered trinomials of $\mathbb{F}_{q^6}[X]$ in even characteristic}
\author{Daniele Bartoli, Giovanni Longobardi, Giuseppe Marino and Marco Timpanella}
\date{}

\DeclareMathOperator{\Tr}{Tr}
\DeclareMathOperator{\N}{N}
\DeclareMathOperator{\PG}{{PG}}

\DeclareMathOperator{\GL}{{GL}}

\DeclareMathOperator{\Gal}{Gal}

\begin{document}
\maketitle

\newtheorem{theorem}{Theorem}[section]
\newtheorem{lemma}[theorem]{Lemma}
\newtheorem{conj}[theorem]{Conjecture}
\newtheorem{remark}[theorem]{Remark}
\newtheorem{cor}[theorem]{Corollary}
\newtheorem{prop}[theorem]{Proposition}
\newtheorem{definition}[theorem]{Definition}
\newtheorem{result}[theorem]{Result}
\newtheorem{property}[theorem]{Property}
\newtheorem{question}[theorem]{Question}

\makeatother
\newcommand{\Prf}{\noindent{\bf Proof}.\quad }
\renewcommand{\labelenumi}{(\alph{enumi})}
\newcommand{\la}{\langle}
\newcommand{\ra}{\rangle}
\newcommand{\G}{\mathrm{\Gamma L}}


\def\B{\mathbf B}
\def\C{\mathbf C}
\def\Z{\mathbf Z}
\def\Q{\mathbf Q}
\def\W{\mathbf W}
\def\a{\mathbf a}
\def\b{\mathbf b}
\def\c{\mathbf c}
\def\d{\mathbf d}
\def\e{\mathbf e}
\def\l{\mathbf l}
\def\v{\mathbf v}
\def\w{\mathbf w}
\def\x{\mathbf x}
\def\y{\mathbf y}
\def\z{\mathbf z}
\def\t{\mathbf t}
\def\cD{\mathcal D}
\def\cC{\mathcal C}
\def\cH{\mathcal H}
\def\cM{{\mathcal M}}
\def\cK{\mathcal K}
\def\cQ{\mathcal Q}
\def\cU{\mathcal U}
\def\cS{\mathcal S}
\def\cT{\mathcal T}
\def\cR{\mathcal R}
\def\cN{\mathcal N}
\def\cA{\mathcal A}
\def\cF{\mathcal F}
\def\cL{\mathcal L}
\def\cP{\mathcal P}
\def\cG{\mathcal G}
\def\cGD{\mathcal GD}

\def\PG{{\rm PG}}
\def\GF{{\rm GF}}
\def\rk{{\rm rk}}

\def\Pg{PG(5,q)}
\def\pg{PG(3,q^2)}
\def\ppg{PG(3,q)}
\def\HH{{\cal H}(2,q^2)}
\def\F{\mathbb F}
\def\Ft{\mathbb F_{q^t}}
\def\P{\mathbb P}
\def\V{\mathbb V}
\def\bS{\mathbb S}
\def\E{\mathbb E}
\def\N{\mathbb N}
\def\K{\mathbb K}
\def\D{\mathbb D}
\def\ps@headings{
 \def\@oddhead{\footnotesize\rm\hfill\runningheadodd\hfill\thepage}
 \def\@evenhead{\footnotesize\rm\thepage\hfill\runningheadeven\hfill}
 \def\@oddfoot{}
 \def\@evenfoot{\@oddfoot}
}
\def\cub{\mathscr C}
\def\cO{\mathcal O}
\def\cur{\mathscr L}
\def\Fqm{{\mathbb F}_{q^m}}
\def\Fq3{{\mathbb F}_{q^3}}
\def\fq{{\mathbb F}_{q}}
\def\Fm{{\mathbb F}_{q^m}}

\begin{abstract}
In recent years, several families of scattered polynomials have been investigated in the literature. However, most of them only exist in odd characteristic. In \cite{CsMZ2018,MMZ}, the authors proved that the trinomial $f_c(X)=X^{q}+X^{q^{3}}+cX^{q^{5}}$ of $\mathbb{F}_{q^6}[X]$ is scattered under the assumptions that $q$ is odd and $c^2+c=1$. They also explicitly observed that this is false when $q$ is even. In this paper, we provide a different set of conditions on $c$ for which this trinomial is scattered in the case of even $q$. Using tools of algebraic geometry in positive characteristic, we show that when $q$ is even and sufficiently large, there are roughly $q^3$ elements $c \in \F_{q^6}$ such that $f_{c}(X)$ is scattered.
Also, we prove that the corresponding MRD-codes and $\mathbb{F}_q$-linear sets of
$\PG(1,q^6)$ are not
equivalent to the previously known ones.
\end{abstract}

\bigskip


\section{Introduction, preliminaries and notations}\label{Section1}

Let $\F_{q^n}$ be the finite field of order $q^n$, $q$ a prime power and let $\sigma : x \in \F_{q^n} \longmapsto
 x^{q^s} \in \F_{q^n}$ be an element of $\mathrm{Gal}(\F_{q^n} \vert \F_q$), with $1 \leq s \leq n-1$ and $\gcd(s,n)=1$. A polynomial 
$$f(X) = \sum_{i=0}^k a_iX^{\sigma^i} \in \F_{q^n}[X],\,k \in \mathbb{N}$$
is called $\sigma$-\textit{linearized polynomial}, or  $q^s$-\textit{ polynomial} over $\F_{q^n}$.
If $a_k \not = 0$, the integer $k$ is called the $\sigma$-\textit{degree} of $f(X)$ and denoted by $\deg_\sigma(f)$. By \cite[Chapter 3]{lidl}, any endomorphism of $\F_{q^n}$, seen as vector space over $\F_q$, can be represented as a $\sigma$-linearized polynomial over $\F_{q^n}$ of $\sigma$-degree less than $n$, and this correspondence is $1$-to-$1$. Precisely, denoted by $\tilde{\mathcal{L}}_{n,q,\sigma}$ or $\tilde{\mathcal{L}}_{n,q,s}$,  the set of the $\sigma$-polynomials over $\F_{q^n}$ with $\sigma$-degree at most $n-1$, the algebraic structure $(\tilde{\mathcal{L}}_{n,q,\sigma},+,\cdot)$, where $+$ is the usual sum between polynomials and $\cdot$ the scalar multiplication, is a vector space over $\F_{q^n}$. When $(\tilde{\mathcal{L}}_{n,q,\sigma},+,\cdot)$ is regarded as a vector space over $\F_q$ its dimension is $n^2$ and it is
isomorphic to $\F_{q}^{n \times n}$. Moreover, $\tilde{\mathcal{L}}_{n,q,\sigma}$ endowed with the product $\circ$ modulo $X^{q^n}-X$ is an algebra over $\F_q$ isomorphic to $\mathrm{End}_{\F_q}(\F_{q^n})$. So, for a given $f(X) \in \tilde{\mathcal{L}}_{n,q,\sigma}$, we can refer to the \textit{kernel} $\ker f$ of $f$, and to the \textit{image} $\mathrm{im\,}f$ of $f$, as the set of roots of $f$ and the set $f(\F_{q^n})$ respectively, and these are $\F_q$-subspaces of $\F_{q^n}$.\\
Let $f(X)= \sum_{i=0}^{n-1}a_iX^{\sigma^i} \in \tilde{\mathcal{L}}_{n,q,\sigma}$ be a $\sigma$-polynomial and let $D_{f,\sigma}$ denote the associated $\sigma$-\textit{Dickson matrix} (or $\sigma$-\textit{circulant matrix} )
\begin{equation*}
D_{f,\sigma}=\begin{pmatrix}
    a_0 & a_1 & \ldots & a_{n-1} \\
    a^\sigma_{n-1} & a^\sigma_0 & \ldots  & a^{\sigma}_{n-2} \\
    \vdots & \vdots & \vdots & \vdots \\
    a_{1}^{\sigma^{n-1}} & a_{2}^{\sigma^{n-1}} & \ldots & a^{\sigma^{n-1}}_0
\end{pmatrix}.
\end{equation*}

The rank of $D_{f,\sigma}$ is the dimension of $\mathrm{im}\,f$, i.e. the \textit{rank} of the $\F_q$-linear map $f(X)$, see \cite{Menichetti,wuliu}.
Very recently, a class of linearized polynomials over a finite field has captured the interest of many mathematicians because of its connections with finite geometry and with
coding theory: the class of scattered polynomials. 
Indeed, J. Sheekey made a breakthrough in the construction of 
matrix codes endowed with rank-metric in \cite[Section 5]{Sh} as we shall briefly retrace it below.\\
A $\sigma$-linearized polynomial $f(X) \in \F_{q^n}[X]$ is called \textit{scattered} if for any $m \in \F_{q^n}$  
$$\dim_{\mathbb{F}_q} (\ker(mX+f(X)) \leq 1,$$ 
or equivalently, if for any $y, z \in \F^*_{q^n}$, $f(y)z-f(z)y=0$
implies that $y$ and $z$ are $\F_q$-linearly dependent. The condition of scatteredness  can be equivalently stated in terms of Dickson
(sub-) matrices as done in \cite[Corollary 3.5]{subresultants} (see also \cite{DicksonZanella}) where,  although it holds for   $q^s$-linearized polynomials, is proven for $s=1$. So, let $f(X) = \sum_{i=0}^{n-1} a_iX^{\sigma^i} \in \tilde{\mathcal{L}}_{n,q,\sigma}$ and consider the matrix 
\begin{equation*}
H=\begin{pmatrix}
    Y & a_1 & \ldots & a_{n-1} \\
    a^\sigma_{n-1} & Y^{\sigma}& \ldots  & a^{\sigma}_{n-2} \\
    \vdots & \vdots & \vdots & \vdots \\
    a_{1}^{\sigma^{n-1}} & a_{2}^{\sigma^{n-1}} & \ldots & Y^{\sigma^{n-1}}
\end{pmatrix}.
\end{equation*}
with $Y$ as 
variable. The determinant of the $(n - m) \times (n - m)$ matrix obtained from $H$ after removing its
first $m$ columns and last $m$ rows is a polynomial $H_m(Y) \in  \F_{q^n} [Y]$. Then $f(X)$ is scattered if and only if $H_0(Y)$ and $H_1(Y)$ have no common roots in $\F_{q^n}$.\\
As mentioned before, scattered polynomials have a connection with a class of codes whose codewords are matrices with entries over $\F_q$. These were introduced by Delsarte in \cite{Delsarte} as $q$-analogs
of the usual linear error correcting codes endowed with the Hamming distance, see also \cite{kshevetskiy_new_2005}. 
More precisely, a \textit{rank-metric cod}e (or RM-code) $\cC$ is a subset of the
set of $m \times n$ matrices  of $\F_{q}^{m\times n}$, equipped with the rank metric
\begin{equation*}
d(A, B)= \rk(A - B) \,\, \textnormal{ for any } A, B \in \F_q^{m \times n}.
\end{equation*}

The \textit{minimum distance} of an RM-code $\cC,$ $|\cC| \geq 2$, is
\begin{equation*}
    d(\cC) = \min_{A,B \in \cC, A \not = B}
d(A, B).
\end{equation*}

If $\cC$ is an $\F_q$-linear subspace of $\F_{q}^{m \times n}$, then $\cC$ is called $\F_q$-\textit{linear}
RM-code. In this paper, we will deal with linear RM-codes where $m = n$. So, the rank-metric codes that we will take in account are $\F_q$-subspaces of $\F_{q}^{n \times n}$ and have square matrices of order $n$ as codewords. Therefore, they might be
regarded as a suitable  $\F_q$-linear subspace of $\tilde{\mathcal{L}}_{n,q,\sigma}$.
In this setting, an $\F_q$-linear rank-metric code $\cC$ is an $\F_q$-subspace of $\tilde{\mathcal{L}}_{n,q,\sigma}$ endowed
with the rank metric 
\begin{equation*}
    d(f_1,f_2)=\rk(f_1-f_2)
\end{equation*}and minimum distance 
\begin{equation*}
  d(\cC) = \min_{f \in \cC \setminus \{0\}} \rk \,f .  
\end{equation*}

In \cite{Delsarte}, it is showed that an $\F_q$-linear code $\cC \leq \tilde{\mathcal{L}}_{n,q,\sigma}$ with minimum distance $d$ have dimension over $\F_q$, $\dim_{q}\mathcal{C}$, at most $n(n - d+1)$. 
If $\dim_q \cC$ meets this bound, then $\cC$ is a called \textit{maximum rank-metric} or MRD-code for short. 
The \textit{adjoint code} of a rank-metric code $\cC$ is \begin{equation*}\hat{\cC} = \{\hat{f} \colon f \in \cC\},
\end{equation*}
where if $f(X)=\sum_{i=0}^{n-1}a_iX^{\sigma^i}$, the polynomial $\hat{f}(X)=\sum_{i=0}a^{\sigma^i}_{n-i}X^{\sigma^i}$ 
denotes the \textit{adjoint} of $f(X)$. Two $\F_q$-linear codes $\cC$ and $\cC'$ are
called \textit{equivalent} if there exist $g(X)$ and $h(X)$ in $\mathcal{\tilde{L}}_{n,q,\sigma}$ of rank $n$ and a field automorphism $\rho$
of $\F_{q^n}$ such that
$$\cC' = \{g \circ f^\rho \circ h: f \in \cC\}$$
where $f^\rho(X)$ stands for $\sum_{i=0}^{n-1}a^\rho_iX^{\sigma^i}$.
Furthermore, $\cC$ and $\cC'$
are \textit{adjointly equivalent} if $\hat{\cC}$ is equivalent to $\cC'$.
Two invariants under equivalence are the idealizers of a code $\cC \subseteq \tilde{\mathcal{L}}_{n,q,\sigma}$. They are
defined as the sets
$$I_L(\cC) = \{\varphi(X) \in \tilde{\mathcal{L}}_{n,q,\sigma} : \varphi \circ f \in  \cC\, \forall f \in \cC\},$$
and
$$I_R(\cC) =    \{\varphi(X) \in \tilde{\mathcal{L}}_{n,q,\sigma} : f \circ \varphi \in \cC\, \forall f \in \cC\},$$
respectively. By \cite[Corollary 5.6]{LTZ2} and \cite[Theorem 3.1]{LZ-standard}, if $\cC$ is a linear MRD code, both idealizers are isomorphic to subfields of $\F_{q^n}$.

Scattered polynomials over $\F_{q^n}$ are linked to particular $\F_q$-linear subspaces of $V=\F_{q^n}\times \F_{q^n}$ which are called  scattered. More precisely, an $\F_q$-subspace $U$ of dimension $k$ of $V$ is  \textit{scattered} if 
\begin{equation*}
    \dim_{\F_q}(U\cap\la{\mathbf{v}\ra_{\F_{q^n}})\leq 1 \, \textnormal{ for each } {\mathbf v}\in V\setminus\{\mathbf 0}\}.
\end{equation*}

By \cite[Theorem 4.3]{BL2000}, the dimension of a scattered subspace $U$ of $V$ is bounded by $n$. If this bound is attained, $U$ is called \textit{maximum scattered}. 
Up to the action of the group $\GL(2,q^n)$, any $\F_q$-subspace $U$ of $V$ of rank $n$ can be written as $U=U_f=\{(x,f(x))\colon x\in\F_{q^n}\}$, for some $f(X) \in\tilde{\cL}_{n,q,\sigma}$.

In \cite{Sh}, the author proved the following result (which have been generalized in \cite[Section 2.7]{Lu2017} and \cite{ShVdV}, see also \cite[Result 4.7]{CsMPZ2019} and \cite{LTZ}).

\begin{result}\label{result:MRDLS}	
An $\F_q$-linear subspace $\cC$ of $\tilde{\mathcal{L}}_{n,q,\sigma}$is an MRD-code with minimum distance $d=n-1$ and with left-idealiser isomorphic to $\F_{q^n}$ if and only if, up to equivalence,
\[\cC=\{ a X + b f(X) \colon a,b, \in \F_{q^n}\}=\la X,f(X)\ra_{\F_{q^n}}\]
for some scattered polynomial $f(X) \in \cL_{n,q,\sigma}$	and the $\F_q$-subspace
\[U_{\cC}=\{(x,f(x)) \colon x\in \F_{q^n}\}\]
is a maximum scattered $\F_q$-subspace of $V=\F_{q^n}\times\F_{q^n}$.
Moreover, two $\F_q$-linear MRD-codes $\cC$ and $\cC'$ of $\mathcal{L}_{n,q,\sigma}$, with minimum distance $d=n-1$ and with left idealisers isomorphic to $\F_{q^n}$, are equivalent if and only if $U_\cC$ and $U_{\cC'}$ are $\Gamma {\mathrm L}(2,q^n)$-equivalent.

\end{result}

So far, the known non-equivalent (under $\Gamma\mathrm{L}(2,q^n)$) maximum scattered $\F_q$-subspaces, which provide the known non-equivalent $\F_q$-linear MRD-codes with left idealiser isomorphic to $\F_{q^n}$, are
\begin{enumerate}
\item $U^{1,n}_s:= \{(x,x^{q^s}) \colon x\in \F_{q^n}\}$, $1\leq s\leq n-1$, $\gcd(s,n)=1$, see \cite{BL2000,CSZ2016};
\item $U^{2,n}_{s,\delta}:= \{(x,\delta x^{q^s} + x^{q^{n-s}})\colon x\in \F_{q^n}\}$, $n\geq 4$, $\mathrm{N}_{q^n/q}(\delta)\notin \{0,1\}$ \footnote{Here and later in the paper, $\mathrm{N}_{q^n/q^\ell}(x)=x^{\frac{q^n-1}{q^\ell-1}}$, where $x \in \F_{q^n}$ and $\ell \mid n$.},  $1 \leq s \leq n-1$ and $\gcd(s,n)=1$, see \cite{LP2001} for $s=1$, \cite{Sh,LTZ} for $s\neq 1$;
\item $U^{3,n}_{s,\delta}:= \{(x,\delta x^{q^s}+x^{q^{s+n/2}})\colon x\in \F_{q^{n}}\}$, $n\in \{6,8\}$, $1 \leq s \leq n-1$ with $\gcd(s,n/2)=1$, $\mathrm{N}_{q^n/q^{n/2}}(\delta) \notin \{0,1\}$, for some conditions on $\delta$ and $q$ see \cite[Theorems 7.1 and 7.2]{CMPZ} and \cite{TZ};
\item $U^{4}_{c}:=\{(x, x^q+x^{q^3}+c x^{q^5}) \colon x \in \F_{q^6}\}$, $q$ odd, $c^2+c=1$, see \cite{CsMZ2018, MMZ}. 
\item $U^5_{h,t,\sigma}=\{(x, \psi_{h,t,\sigma}(x) \colon x \in \F_{q^{n}}\}$, where \begin{equation*}
    \psi_{h,t,\sigma}(x)=  x^{\sigma}+x^{\sigma^{t-1}}-h^{1-\sigma^{t+1}}x^{\sigma^{t+1}}+h^{1-\sigma^{2t-1}}x^{\sigma^{2t-1}},
\end{equation*} $n=2t$, $q$ odd, $\sigma$ a generator of $\Gal(\F_{q^n}|\F_q)$ and $h$ is any element of $\F_{q^n}$ such that $h^{q^t+1}=-1$, see \cite{NSZ}; see also  \cite{bartoli_new_maximum_2020} for $t=3$ and \cite{LMTZ} for $t>3$, both with $\sigma: x \in \F_{q^n} \longrightarrow x^q \in \F_{q^n}$. Moreover, for $h \in \F_{q}$, the scattered polynomial $\psi_{h,t,\sigma}$ is widely studied in \cite{bartoli_new_maximum_2020} for $t=3$ and \cite{longobardi_familyOfLinearMRD_2021} $t >3 $. More precisely, the polynomial
\begin{equation*}
    \psi^{(k)}(x)= \frac{1}{2}\left(x^q+x^{q^{t-k}} -x^{q^{t+k}}+x^{q^{2t-k}}\right)
\end{equation*} turns out to be scattered when $ 1 \leq k \leq t$ and 
	\begin{itemize}
		\item [-] $t$ is even, $\gcd(k,t)=1$, or
		\item [-] $t$ is 
 odd, $\gcd(k,2t)=1$, and $q\equiv 1\pmod{4}$. \end{itemize}

\end{enumerate}
When $q$ is even, the only families of $\F_q$-subspaces of $\F_{q^n}^2$ producing MRD-codes are the families $(a)$, $(b)$ and $(c)$ with $n=6$ in the previous list. Motivated by these few known families,
in this article, we investigate a class of linearized trinomials over $\F_{q^6}$ for $q$ \textbf{even}.  More precisely, let
$$f_{c,s}(X)=X^{q^s}+X^{q^{3s}}+cX^{q^{5s}} \in \tilde{\mathcal{L}}_{6,q,s}, $$
with  $s \in \{1,5\}$. In \cite{CsMZ2018,MMZ}, the authors proved that trinomial of $\mathbb{F}_{q^6}[X]$ as above is scattered under the assumptions that $q$ is odd and $c^2+c=1$. They also explicitly observed that this is false when $q$ is even. In Section \ref{presentazione}, we shall give, for $q$ even and large enough,  some conditions on $c \in \F_{q^6}$ in order to have $f_{c,s}(X)$ scattered or, equivalently,
the subspace 
\begin{equation}\label{family}
    U_{c,s}=\{(x, f_{c,s}(x)) \colon x \in \F_{q^{6}}\} 
\end{equation}
scattered.  More precisely, as stated in Theorem \ref{maintheorem}, we will show that for sufficiently large  $q$ there are roughly  $q^3$ elements $c \in \F_{q^6}$ such that $f_{c,s}(X)$ is scattered. Later, in Section \ref{MRDcode}, the family of MRD-codes $\mathcal{D}_{c,s}$ associated with the scattered subspaces $U_{c,s}$ are introduced and the equivalence issue among them is solved.   Finally, in Section \ref{linearset}, the connections to maximum scattered linear sets of
the finite projective line are investigated. We prove that  the linear sets in $\PG(1, q^6)$ arising from the subspaces $U_{c,s}$
are not $\mathrm{P \Gamma L}$-equivalent to the previously known linear sets.

\section{A family of maximum scattered of $\F_{q^6}^2$ for $q$ even and large enough} \label{presentazione}

From now on, 
let $q$ be \textbf{a power of $2$} and let $\F_{q^6}$ be the finite field with $q^6$ elements. Consider the $6$-dimensional $\F_q$-subspace of $\F_{q^6}^2$
\[
U_{a,b,c}=\{(x,ax^{q}+bx^{q^{3}}+cx^{q^{5}}) \colon x \in \F_{q^{6}}\},
\] with $a,b,c\in\F_{q^6}^*$. Note that, up to the action of the group $\mathrm{GL}(2,q^6)$, we can always assume $a=1$. Therefore, we can restrict our attention to
\begin{equation}\label{Ubc}
U_{b,c}:=U_{1,b,c}=\{(x, x^{q}+bx^{q^{3}}+cx^{q^{5}}) \colon x \in \F_{q^{6}}\}.
\end{equation}
While the family $\{U_{b,c}\ : \ b,c \in \mathbb{F}_{q^6}\}$ remains too expansive for exhaustive examination, it is plausible that certain subspaces within it are equivalent to already known scattered subspaces. Consequently, our objective is to narrow down parameter selections and uncover new instances of scattered subspaces.
To pursue this goal, we initially establish the following result.
\begin{lemma}\label{prop:equiv}
If $c\ne b^{q^2+1}$, the $\F_q$-subspace $U_{b,c}$ is not equivalent to the $\F_q$-subspaces in the families $(a)$, $(b)$, $(c)$.
\end{lemma}

\begin{proof}
The $\F_q$-subspace $U_{b,c}$ is $\G(2,q^6)$-equivalent to $U_s^{1,6}$, with $s=1,5$, if and only if there exist $A,B,C,D \in \F_{q^6}$, with $AD-BC\ne 0$ and $\rho\in \mathrm{Aut}(\F_{q^6})$ such that
\[\left(
\begin{array}{cc}      
A & B \\
C & D \\
\end{array}\right)
 \left( \begin{array}{c}
x\\
 x^{q^s} \\
\end{array}
 \right)^\rho=
\left(
 \begin{array}{c}

 z\\
                    z^{q}+bz^{q^{3}}+cz^{q^{5}}\\
                  \end{array}
                \right).
                \]
By computations, from the previous equality, one gets a polynomial identity in the variable $x^{\rho}$, and comparing the coefficients of $x^{\rho q^{2}}$ and $x^{\rho q^{4}}$, we get in both cases $A=B=0$, a contradiction.
In the same way, the $\F_q$-subspace $U_{b,c}$ is $\G(2,q^6)$-equivalent to $U_{1,\delta}^{2,6}$ if and only if there exist $A,B,C,D \in \F_{q^6}$, with $AD-BC\ne 0$ and $\rho\in \mathrm{Aut}(\F_{q^6})$ such that
\begin{equation*}
    \left(
                  \begin{array}{cc}
                    A & B \\
                    C & D \\
                  \end{array}
                \right)
                 \left(
                  \begin{array}{c}
                    x\\
                    \delta x^q+x^{q^5} \\
                  \end{array}
                \right)^\rho=
                \left(
                  \begin{array}{c}
                    z\\
                    z^{q}+bz^{q^3}+cz^{q^{5}}\\
                  \end{array}
                \right).
                \end{equation*}
                
From the previous equality we get a polynomial identity in the variable $x^{\rho}$. Comparing the coefficients of $x^{\rho q}$ and $x^{\rho q^3}$, it follows that $A=D=0$, whereas comparing the  coefficients of $x^{\rho q^2}$ and $x^{\rho q^4}$ we get
\begin{equation*}\delta^{\rho q}B^q+bB^{q^3}=0 \quad \text{ and } \quad b\delta^{\rho q^3}B^{q^3}+cB^{q^5}=0.
\end{equation*}

By raising the first equation to the $q^2$-th power and taking $c\ne b^{q^2+1}$ into account, we get $B=0$, a contradiction. Arguing in the same way it is possible to show that $U_{b,c}$ is not  $\G(2,q^6)$-equivalent to $U_{5,\delta}^{2,6}$ when $c\ne b^{q^2+1}$.

Finally, the $\F_q$-subspace $U_{b,c}$ is $\G(2,q^6)$-equivalent to $U_{s,\delta}^{3,6}$, $S \in \{1,5\}$, if and only if there exist $A,B,C,D \in \F_{q^6}$, with $AD-BC \ne 0$ and $\rho\in \mathrm{Aut}(\F_{q^6})$ such that
\[\left(
                  \begin{array}{cc}
                    A & B \\
                    C & D \\
                  \end{array}
                \right)
                 \left(
                  \begin{array}{c}
                    x\\
                    \delta x^{q^s}+x^{s+3} \\
                  \end{array}
                \right)^\rho=
                \left(
                  \begin{array}{c}
                    z\\
                    z^q+bz^{q^3}+cz^{q^5}\\
                  \end{array}
                \right).
                \]
Therefore, a polynomial identity in the variable $x^{\rho}$ follows, and comparing the coefficients of $x^{\rho (s+1)}$ and $x^{\rho (s+2)}$ we get $A=B=0$, a  contradiction again.
\end{proof}
For the remainder of this section, we investigate the family of $\F_q$-subspaces of $\F_{q^6}\times \F_{q^6}$
\[
U_{c}:=U_{1,c}=\{(x, x^q+x^{q^{3}}+cx^{q^{5}}) \colon x \in \F_{q^{6}}\}
 \textnormal{ with } c \in \F^*_{q^6}\]
and we shall give some sufficient conditions on $c \in \F_{q^6}$ for $U_c$ to be scattered. Note that, by Lemma \ref{prop:equiv}, if $c\neq 0,1$, the subspace $U_{c}$ is not $\Gamma \mathrm{L}(2,q^6)$-equivalent to the previously known ones.\\
Consider the following polynomials in $\F_q[X]$
\begin{eqnarray}\label{ccq4}
F_1(X)=X^{q^4+2q^3+q^2+q+1}+X^{q^4+q^3+q^2+1}+X^{q^3+q+1} + X +X^{2q^3+q^2+q} +X^{q^3+q^2+q}+ X^{q^3} + 1,
\end{eqnarray}
\begin{eqnarray}\label{prova3}
F_2(X)=&X^{q^3+q^2+2q+2} + X^{q^2+q+2}+ X^{q^3+2q^2+2q+1} + X^{2q^3+2q^2+q+1} + X^{q^2+q+1}\\&+ X^{q^3+q+1}+ X^{q^3+q^2+1}+ X + X^{2q^3+q^2+q} +  X^{q^3+q^2+q}+ X^{q^3} +1,\nonumber
\end{eqnarray}
and
\begin{eqnarray}\label{prova31}
&&F_3(X)=X^{q^2+q+1}+1.
\end{eqnarray}
As a first step, we prove the following result. \begin{lemma}\label{Prop23}
Let $c \in \F_{q^6}$ and let $F_1(X)$, $F_2(X)$ and $F_3(X)$ be the polynomials in $\F_q[X]$ defined in \eqref{ccq4}, \eqref{prova3} and \eqref{prova31}. If $F_1(c)=F_2(c)=0$, and $F_3(c) \not =0$, then $c$ does not belong to $\mathbb{F}_{q^3}$.
\end{lemma}
\begin{proof}
By the way of contradiction, assume $c \in  \mathbb{F}_{q^3}$. Then, $F_1(c)=F_2(c)=0$ yield
\begin{eqnarray*}
(c^{q^2+q+1}+1)(c^{q+2}+1)&=&0,\\
(c^{q^2+q+1}+1)(c^{q+2}+c^{q^2+2}+c^{q^2+q+1}+1)&=&0.
\end{eqnarray*}

As $c^{q^2+q+1}+1\neq 0$ by $F_3(c)\neq 0$, $c^{q+2}=1$ and $c^{q+2}+c^{q^2+2}+c^{q^2+q+1}+1=0$ must hold, whence $c^{q^2+2}+c^{q^2+q+1}=c^{q^2+1}(c+c^q)=0$. So, $c\in \mathbb{F}_q$ and $c^3=1$, a contradiction with $F_3(c) \not=0$.
\end{proof}

Now, recall that $U_{c}$, $c \in \F_{q^6}$, is scattered if and only if for every $m\in \F_{q^6}$ the Dickson matrix 
\begin{equation}\label{Matrice1}
   D_{c}^{(m)} =
   \begin{pmatrix}
   m & 1 & 0 & 1 & 0 & c\\
   c^q & m^q & 1 & 0 & 1 & 0\\
   0 & c^{q^2} & m^{q^2} & 1 & 0 & 1\\
   1 & 0 & c^{q^3} & m^{q^3} & 1 & 0\\
   0 & 1 & 0 & c^{q^4} & m^{q^4} & 1\\
   1 & 0 & 1 & 0 & c^{q^5} & m^{q^5}\\
   \end{pmatrix}
\end{equation}
has rank at least $5$. By \cite[Corollary 3.5]{subresultants}, this is equivalent to requiring that for every $m\in \F_{q^6}$  the determinants
\begin{equation*}
   \begin{vmatrix}
   m & 1 & 0 & 1 & 0 & c\\
   c^q & m^q & 1 & 0 & 1 & 0\\
   0 & c^{q^2} & m^{q^2} & 1 & 0 & 1\\
   1 & 0 & c^{q^3} & m^{q^3} & 1 & 0\\
   0 & 1 & 0 & c^{q^4} & m^{q^4} & 1\\
   1 & 0 & 1 & 0 & c^{q^5} & m^{q^5}\\
   \end{vmatrix}
   \end{equation*}
   and 
   \begin{equation*}
   \begin{vmatrix}
     1 & 0 & 1 & 0 & c\\
    m^q & 1 & 0 & 1 & 0\\
    c^{q^2} & m^{q^2} & 1 & 0 & 1\\
    0 & c^{q^3} & m^{q^3} & 1 & 0\\
    1 & 0 & c^{q^4} & m^{q^4} & 1\\
   \end{vmatrix}
\end{equation*}
don't vanish simultaneously. 

Then, consider the polynomials in $\F_{q^6}[X,Y,Z,T,U,V]$
\begin{equation*}
 p(X,Y,Z,T,U,V)=   \begin{vmatrix}
   X & 1 & 0 & 1 & 0 & c\\
   c^q & Y & 1 & 0 & 1 & 0\\
   0 & c^{q^2} & Z & 1 & 0 & 1\\
   1 & 0 & c^{q^3} & T & 1 & 0\\
   0 & 1 & 0 & c^{q^4} & U & 1\\
   1 & 0 & 1 & 0 & c^{q^5} & V\\
   \end{vmatrix}
\,\,\text{  and  } \,\,
q(X,Y,Z,T,U,V)=
   \begin{vmatrix}
    1 & 0 & 1 & 0 & c\\
    Y & 1 & 0 & 1 & 0\\
    c^{q^2} & Z & 1 & 0 & 1\\
    0 & c^{q^3} & T & 1 & 0\\
    1 & 0 & c^{q^4} & U & 1\\
   \end{vmatrix},
\end{equation*}
and denote by $\Psi$ the automorphism of the polynomial ring $\F_{q^6}[X,Y,Z,T,U,V]$ that maps the polynomial
$$h(X,Y,Z,T,U,V)=\sum_{i_0,i_1,i_2,i_3,i_4,i_5} \alpha_{i_0i_1i_2i_3i_4i_5}X^{i_0}Y^{i_1}Z^{i_2}T^{i_3}U^{i_4}V^{i_5}$$
to the polynomial $$\Psi(h)(X,Y,Z,T,U,V)=\sum_{i_0,i_1,i_2,i_3,i_4,i_5}\alpha_{i_0i_1i_2i_3i_4i_5}^qX^{i_5}Y^{i_0}Z^{i_1}T^{i_2}U^{i_3}V^{i_4}.$$
Also, for a polynomial
\begin{equation*}
    h(X)=\sum_{ \underset{0\leq j \leq 5}{0 \leq i_j < q}} a_{i_0i_1i_2i_3i_4i_5}X^{i_0}X^{i_1q}X^{i_2q^2}X^{i_3q^3}X^{i_4q^4}X^{i_5q^5}\in \F_{q}[X]
\end{equation*}
we let
\begin{equation*}\label{multivariate}
    \Phi(h)(X_0,X_1,X_2,X_3,X_4,X_5)=\sum_{ \underset{0\leq j \leq 5}{0 \leq i_j < q}} a_{i_0i_1i_2i_3i_4i_5} X_0^{i_0}X_1^{i_1}X_2^{i_2}X_3^{i_3}X_4^{i_4}X_5^{i_5}\in \F_q[X_0,X_1,X_2,X_3,X_4,X_5].
\end{equation*}
Note that an element $c \in \F_{q^6}$ is a root of $h(X)$ if and only if $(c,c^q,c^{q^2},c^{q^3},c^{q^4},c^{q^5})$ is a root of the polynomial $\Phi(h)(X_0,X_1,X_2,X_3,X_4,X_5)$. Then, as consequence of \cite[Corollary 3.5]{subresultants}, we can state the following.

\begin{lemma}\label{lemma-cons}
If, for a certain choice of $c \in \F_{q^6}$, neither $(0, 0, 0, 0, 0, 0)$ nor the elements of $(\F_{q^6}^*)^6$
are solutions of the system
\begin{equation}\label{system1}
\begin{cases}
p(X,Y,Z,T,U,V)=0,\\
q_0(X,Y,Z,T,U,V):=q(X,Y,Z,T,U,V) =0,\\
q_1(X,Y,Z,T,U,V):=\Psi(q_0)(X,Y,Z,T,U,V)=0,\\
q_2(X,Y,Z,T,U,V):=\Psi(q_1)(X,Y,Z,T,U,V) =0,\\
q_3(X,Y,Z,T,U,V):=\Psi(q_2)(X,Y,Z,T,U,V) =0,\\
q_4(X,Y,Z,T,U,V):=\Psi(q_3)(X,Y,Z,T,U,V) =0,\\
q_5(X,Y,Z,T,U,V):=\Psi(q_4)(X,Y,Z,T,U,V) =0,\\
\end{cases}  
\end{equation}
 then $U_{c}$ is maximum scattered.
\end{lemma}
In order to show that $U_c$ is scattered for some $c \in \F_{q^6}$, one shall calculate some resultants of the equations that appear in the system above. This argument is based on the fact that taking the resultant $R$ of two polynomials with respect to any variable in a set $\mathcal{S}$ of polynomials, the set $\mathcal{S} \cup \{R\}$ admits the same set of common roots as the polynomials in $\mathcal{S}$.\\
By computation, one gets 
\begin{eqnarray*}
  p(X,Y,Z,T,U,V)&=&c^{q^5+q^4+q^3+q^2+q+1} +c^{q^4+q^2+q+1}+ c^{q^2+q+1}TU + c^{q^5+q^3+q+1} + c^{q^5+q+1}ZT + c^{q+1}\\&&
  + c^{q^4+q^3+q^2+1}  + c^{q^5+q^4+q^2+1} + c^{q^2+1}TU + 
    c^{q^3+1}YU + c^{q^3+1} + c^{q^5+q^4+1}YZ \\&&
  + c^{q^4+1}YZ+ + c^{q^5+1} + cYZTU + cYU + cZT + c^{q^5+q^3+q^2+q} + c^{q^3+q^2+q}UV \\&&
  + c^{q^2+q}  + c^{q^5+q^4+q^3+q}+ + c^{q^3+q}UV + c^{q^4+q}ZV + 
    c^{q^4+q} + c^{q^5+q}ZT + c^{q}ZTUV \\&&  
  + c^qZV + c^qTU + c^{q^4+q^3+q^2}XV++ c^{q^3+q^2} +
    c^{q^4+q^2}XV + c^{q^5+q^2}XT + c^{q^5+q^2}\\&&
  + c^{q^2}XTUV + c^{q^2}XT + c^{q^2}UV+ + c^{q^5+q^4+q^3}XY + 
    c^{q^4+q^3} + c^{q^5+q^3}XY \\&&+ c^{q^3}XYUV
  +c^{q^3}XV + c^{q^3}YU + c^{q^5+q^4} + c^{q^4}XYZV + c^{q^4}XY + 
   c^{q^4}ZV\\&&
  + c^{q^5}XYZT + c^{q^5}XT + c^{q^5}YZ+XYZTUV + XYTU + XY +
    XZTV\\&&
 + XT + XV + YZUV + YZ + YU + ZT + ZV + TU + UV,
\end{eqnarray*}
and 
\begin{eqnarray*}
  q(X,Y,Z,T,U,V)&=& c^{q^4+q^3+q^2+1}+ c^{q^4+q^2+1} + c^{q^2+1}TU + c^{q^3+1}YU + c^{q^3+1} + c^{q^4+1}YZ + cYZTU\\&&+ cZT + c + c^{q^3+q^2} + c^{q^2} + c^{q^4+q^3} + c^{q^3}YU +c^{q^4} + YZ + ZT + TU.
\end{eqnarray*}
Consider the following resultants:
\begin{eqnarray*}
    r_1(X,Y,Z,T,V)&=&{\rm Res}(p,q_4,V),\\
    r_2(X,Y,Z,T,V)&=&{\rm Res}(r_1,q_0,U),\\
    r_3(X,Y,Z,T,V)&=&{\rm Res}(r_2,q_5,T).
\end{eqnarray*}
Then, a MAGMA aided computation shows that
\begin{eqnarray}\label{r3}
r_3(X,Y,Z,T,U,V)&=&(c^{q^5+q^2} + c^{q^2} + c^{q^5}YZ + c^{q^5} + 1)(\bar{r}_3(X,Y,Z,T,U,V))^2,
\end{eqnarray}
where
\begin{eqnarray*}
\bar{r}_3(X,Y,Z,T,U,V)&=&A_1(c)Z+A_1(c)^qX+ A_2(c) XYZ +A_3(c)YZ^2+(c^{q^2+q+1}+1)^{q^4}XY^2Z^2 \\&& +
A_4(c)X^2Y+(c^{q^2+q+1}+1)^{q^3}X^2Y^2Z,
\end{eqnarray*}
with
\begin{eqnarray*}
A_1(c) &=& c^{q^5+q^4+q^3+q^2+2q+1} + c^{q^5+q^3+q^2+2q+1} + c^{q^2+2q+1} + c^{2q+1}+c^{q^4+q^3+q^2+q+1}+ c^{q^3+q^2+q+1}\\&&
+ c^{q^5+q^4+q^2+q+1}+ c^{q^5+q^2+q+1}+c^{q^4+q^2+1}+c+c^{q^2+2q}+c^{q^5+q^4+q^3+2q}+c^{q^5+q^3+2q}\\&&+c^{2q}+c^{q^4+q^3+q}+c^{q^3+q}+c^{q^5+q^4+q}+c^{q^5+q}+c^{q^2}+c^{q^4},
\end{eqnarray*}
\begin{eqnarray*}
A_2(c) &=& c^{q^5+q^4+q^3+q^2+q+1}+c^{q+1}+c^{q^4+q^2+1}+c^{q^5+1}+c^{q^2+q}+c^{q^5+q^3+q}+c^{q^3+q^2}+c^{q^4+q^3}+c^{q^5+q^4}+1,\\&&
\end{eqnarray*}
\begin{eqnarray*}
A_3(c) &=& c^{q^5+q^4+q+1} +c^{q^5+q+1}+c^{q+1}+c^{q^4+1}+c^{q}+1,
\end{eqnarray*}
and
\begin{eqnarray*}
A_4(c) &=& c^{q^5+q^4+q^3+q^2}+ c^{q^4+q^3+q^2}+ c^{q^3+q^2}+c^{q^2}+c^{q^5+q^3}+1.
\end{eqnarray*}

As we shall prove in the next result, particular choices of $c\in \mathbb{F}_{q^6}^*$ yield a useful factorization of the polynomial $\bar{r}_3(X,Y,Z,T,U,V)$. This will enable us to investigate the solutions of System \eqref{system1}.

\begin{prop}\label{Prop:fattorizzazione}
Let $c \in \F_{q^6}\setminus \F_{q^2}$ and let $F_1(X)$, $F_2(X)$ and $F_3(X)$ be the polynomials belonging to $\F_q[X]$ defined in \eqref{ccq4},\eqref{prova3} and \eqref{prova31}. If $F_1(c)=F_2(c)=0$ and $F_3(c) \not =0$. Then $\bar{r}_3(X,Y,Z,T,U,V)$ factorizes as

\begin{equation}\label{fattorizzazione}
(\alpha X+ \alpha^q Z)(XY+\beta)(YZ+\gamma),
\end{equation}
where 
\begin{equation}\label{alpha}
    \alpha=(c^{q^2+q+1}+1)^{q^3},
\end{equation}
\begin{equation}\label{beta}
\beta=\frac{c^{q^5+q^4+q+1}+c^{q^5+q+1}+ c^{q+1} +c^{q^4+1} + c^q + 1}{\alpha^q},
\end{equation}
and
\begin{equation}\label{gamma}
\gamma=\frac{c^{q^5+q^4+q^3+q^2} +c^{q^4+q^3+q^2}+c^{q^3+q^2}+ c^{q^2} + c^{q^5+q^3}+ 1}{\alpha}.
\end{equation}
\end{prop}
\begin{proof}
Let $c \in \F_{q^6}\setminus \F_{q^2}$ and note that
$c^{2q^3+q^2+q+1}+c^{q^3+q^2+1}=c^{q^3+q^2+1}(c^{q^3+q}+1)\neq 0$. Indeed, if $c^{q^3+q}+1=0$, then $c^{q^2+1}=1$ and $c^{q^4}=c$, a contradiction.
Since 
\begin{equation}\label{F1(c)}
F_1(c)=(c^{2q^3+q^2+q+1}+c^{q^3+q^2+1})c^{q^4}+c^{q^3+q+1} + c +c^{2q^3+q^2+q} +c^{q^3+q^2+q}+ c^{q^3} + 1=0,
\end{equation}
one has
\begin{equation}\label{ccq4_2}
c^{q^4}=\frac{c^{q^3+q+1} + c +c^{2q^3+q^2+q} +c^{q^3+q^2+q}+ c^{q^3} + 1}{c^{q^3+q^2+1}(c^{q^3+q}+1)}=\phi(c).
\end{equation}

By raising to the $q$-th power, Equation \eqref{ccq4_2} also allows to write $c^{q^5}$ in terms of $c^{q^3},c^{q^2},c^{q}$ and $c$, namely
\begin{equation}\label{ccq5}
c^{q^5}=\frac{\phi(c)c^{q^2+q} + c^q +(\phi(c))^2c^{q^3+q^2} +\phi(c)c^{q^3+q^2}+ \phi(c) + 1}{\phi(c)c^{q^3+q}(\phi(c)c^{q^2}+1)}.
\end{equation}
By \eqref{ccq4_2} and since $c$ does not belong to a proper subfield of $\F_{q^6}$, the denominator of \eqref{ccq5} cannot be zero.
To prove that $\bar{r}_3(X,Y,Z,T,U,V)$ factorizes as in \eqref{fattorizzazione}, we first combine Equations \eqref{ccq4_2} and \eqref{ccq5} with 
$$
\bar{r}_3(X,Y,Z,T,U,V)+(\alpha X+ \alpha^q Z)(XY+\beta)(YZ+\gamma)=g(X,Y,Z).
$$
where $\alpha,\beta$ and $\gamma$ are taken as in the statement. By  MAGMA computations, it can be checked that
\begin{equation*}
\begin{split}
&F_2(c)=c^{q^3+q^2+2q+2} + c^{q^2+q+2}+ c^{q^3+2q^2+2q+1} + c^{2q^3+2q^2+q+1} + c^{q^2+q+1}\\
&+ c^{q^3+q+1}+ c^{q^3+q^2+1}+ c + c^{2q^3+q^2+q} +  c^{q^3+q^2+q}+ c^{q^3} +1
\end{split}
\end{equation*}
divides the coefficients of $g(X,Y,Z)$. So, since $F_2(c)=0$, the claim follows.
\end{proof}

Now, consider the following polynomials in $\F_q[X]$
\begin{equation}\label{F4}
\begin{split}
F_4(X)=&X^{2q^2+q+4}+X^{3q^2+2q+3}+X^{q^2+2q+3}+X^{3q^2+q+3}+X^{q^2+q+3}+X^{2q^2+3}+X^{q^2+3}\\
&+X^{4q^2+q+2}+X^{q+2}+
X^{3q^2+2}+
X^{2}+X^{3q^2+2q+1}+X^{q^2+2q+1}+X^{3q^2+q+1}+X^{q^2+q+1}\\
&+X^{3q^2+1}+X+X^{2q^2+q}+X^{2q^2}+X^{q^2}
\end{split}
\end{equation}
and \begin{equation}\label{F5}
    F_5(X)=X^{2q^2+q+2} + X^{q^2+q+2} + X^{2q^2+2} + X^{2q^2+q+1} + X^{q^2+q+1} + 1.
\end{equation}

From now on, we will use the notation
\begin{equation*}
\mathfrak{C}=\{ c \in \F_{q^6} \setminus \F_{q^2} \colon F_1(c)=F_2(c)=0 \textnormal{ and } F_3(c)F_4(c)F_5(c) \not =0 \}.
\end{equation*}

In order to exploit the above mentioned factorization of $\bar{r}_3(X,Y,Z,T,U,V)$ to study System \eqref{system1}, we must first prove that the set $\mathfrak{C}$ is not empty for $q$ large enough. To do this, we will use the following result due to Cafure and Matera on the number of affine points of an algebraic variety.

\begin{theorem}\label{Th:CafureMatera}\cite[Theorem 7.1]{MR2206396}
Let $\mathcal{V}\subset \mathrm{AG}(N,\overline{\mathbb{F}}_q)$ be an absolutely irreducible variety defined over $\mathbb{F}_q$ of dimension $n$ and degree $d$. If $q>2(n+1)d^2$, then the following estimate holds:
$$|\#(\mathcal{V}\cap \mathrm{AG}(N,\F_q))-q^n|\leq (d-1)(d-2)q^{n-1/2}+5d^{13/3} q^{n-1}.$$
\end{theorem}

In particular, the following corollary holds.
\begin{cor}\label{Th:AsyCafureMatera}
The number of points in $\mathrm{AG}(N,\F_q)$ of an (absolutely) irreducible variety $\mathcal{V} \subset \mathrm{AG}(N,\overline{\mathbb{F}}_q)$ defined over $\mathbb{F}_q$, of dimension $n$ and degree $d$ such that $q>6d^{13/3}$ is $q^n+O(q^{n-1/2})$.
\end{cor}

The investigation of suitable algebraic varieties, together with the application of Theorem \ref{Th:CafureMatera}, allows us to prove the following result.

\begin{theorem}\label{scelteperc}
If $q$ is large enough, the set $$\mathfrak{C}=\{c \in \F_{q^6}\setminus \F_{q^2} \colon F_1(c)=F_2(c)=0 \text{ and } F_3(c)F_4(c)F_5(c) \not =0 \}$$
is not empty and its size is $q^3+O(q^{5/2})$.
\end{theorem}
\begin{proof}
Let $\xi,\xi^q,\xi^{q^2},\xi^{q^3},\xi^{q^4},\xi^{q^5}$ be an $\mathbb{F}_q$-normal basis of $\mathbb{F}_{q^6}$ and write $c\in \mathbb{F}_{q^6}$ as 
$$c=c_0\xi+c_1\xi^q+c_2\xi^{q^2}+c_3\xi^{q^3}+c_4\xi^{q^4}+c_5\xi^{q^5},$$
with $c_0,c_1,c_2,c_3,c_4,c_5\in \mathbb{F}_q$.
Then, conditions $F_1(c)=F_2(c)=0$ can be rewritten as
$$f_1^{(0)}(c_0,c_1,c_2,c_3,c_4,c_5)=0 \textrm{ and } f_2^{(0)}(c_0,c_1,c_2,c_3,c_4,c_5)=0.$$

Let $f_i^{(j)}=(f_i^{(0)})^{q^j}$, for $i=1,2$, $j=0,\ldots,5$.
We will prove that the $\mathbb{F}_q$-rational variety $\mathcal{V} \subset\mathrm{AG}(6,\overline{\mathbb{F}}_q)$ defined by the set of polynomials $\{f_i^{(j)}(Z_0,\ldots,Z_5) \colon i=1,2 \textnormal{ and } j=0,\ldots,5\}$ is absolutely irreducible and has dimension $3$.
To this aim, we consider the following change of variables
\begin{equation*}
    \begin{pmatrix}
        X_0 \\
        X_1 \\
        X_2 \\
    X_3 \\
        X_4 \\
        X_5 \\
    \end{pmatrix}=
    \begin{pmatrix}
	\xi&\xi^q&\xi^{q^2}&\xi^{q^3}&\xi^{q^4}&\xi^{q^5}\\
	\xi^q&\xi^{q^2}&\xi^{q^3}&\xi^{q^4}&\xi^{q^5}&\xi\\
	\xi^{q^2}&\xi^{q^3}&\xi^{q^4}&\xi^{q^5}&\xi&\xi^q\\
	\xi^{q^3}&\xi^{q^4}&\xi^{q^5}&\xi&\xi^q&\xi^{q^2}\\
	\xi^{q^4}&\xi^{q^5}&\xi&\xi^q&\xi^{q^2}&\xi^{q^3}\\
	\xi^{q^5}&\xi&\xi^q&\xi^{q^2}&\xi^{q^3}&\xi^{q^4}\\
	\end{pmatrix}
 \begin{pmatrix}
 Z_0 \\
 Z_1 \\
 Z_2 \\
 Z_3 \\
 Z_4 \\
 Z_5 \\
 \end{pmatrix},
\end{equation*}
where the matrix above is non-singular.
Then, $\mathcal{V}$ is $\mathbb{F}_{q^6}$-projectively equivalent to the variety $\mathcal{U}$ defined by the equations $$g_1^{(j)}(X_0,X_1,X_2,X_3,X_4,X_5)=0 \textrm{ and } g_2^{(j)}(X_0,X_1,X_2,X_3,X_4,X_5)=0, \quad j=0,\ldots,5,$$
where 
\begin{eqnarray*}
g_1^{(0)}(X_0,X_1,X_2,X_3,X_4,X_5)&=&X_0X_1X_2X_3^2X_4+X_0X_2X_3X_4+X_3X_1X_0+X_0\\
&&+X_3^2X_2X_1+X_3X_2X_1+X_3+1,\\
g_2^{(0)}(X_0,X_1,X_2,X_3,X_4,X_5) &=&X_3X_2X_1^2X_0^2+X_3X_2^2X_1^2X_0+X_3^2X_2^2X_1X_0+X_2X_1X_0^2+X_3^2X_2X_1\\
&&+X_2X_1X_0+X_3X_1X_0+X_3X_2X_0+X_3X_2X_1+X_0+X_3+1.
\end{eqnarray*}
and $g_i^{(j)}(X_0,X_1,X_2,X_3,X_4,X_5)=\Psi^{(j)}(g_i^{(0)})(X_0,X_1,X_2,X_3,X_4,X_5)$,  $i=1,2$, $j=0,\ldots,5$. Note that $g_1^{(0)}=\Phi(F_1)$ and $g_2^{(0)}=\Phi(F_2)$.
Our aim is to prove that $\mathcal{U}$
has dimension $3$. 

By $g_1^{(0)}(X_0,X_1,X_2,X_3,X_4,X_5)=0$ and $g_1^{(1)}(X_0,X_1,X_2,X_3,X_4,X_5)=0$, one gets 
$$X_4=\phi_1(X_0,X_1,X_2,X_3)=\frac{X_3 X_1X_0 + X_0 + X_3^2X_2X_1 + X_3 X_2X_1 + X_3 + 1}{X_0X_1 X_2X_3^2 + X_0X_2X_3 }\in \mathbb{F}_q(X_0,X_1,X_2,X_3)$$
and 
$$X_5=\phi_2(X_0,X_1,X_2,X_3)=\phi_1(X_1,X_2,X_3,\phi_1(X_0,X_1,X_2,X_3))\in \mathbb{F}_q(X_0,X_1,X_2,X_3).$$

Now, the numerator of  
$$g_1^{(2)}(X_0,X_1,X_2,X_3,\phi_1(X_0,X_1,X_2,X_3),\phi_2(X_0,X_1,X_2,X_3))$$
is divisible by $g_2^{(0)}(X_0,X_1,X_2,X_3,X_4,X_5)$  and so $g_2^{(k)}(X_0,X_1,X_2,X_3,X_4,X_5)$ divides the numerator of $$g_1^{(k+2)}(X_0,X_1,X_2,X_3,\phi_1(X_0,X_1,X_2,X_3),\phi_2(X_0,X_1,X_2,X_3)) \textnormal{ for } k=1,2,3.$$ This shows that  $\mathcal{U}$ is defined by the set of polynomials $$\{g_1^{(\ell)}(X_0,X_1,X_2,X_3,X_4,X_5),g_2^{(j)}(X_0,X_1,X_2,X_3,X_4,X_5) \colon \ell=0,1 \textnormal{ and } j=0,\ldots,5\}.$$
Now, consider $g_2^{(1)}(X_0,X_1,X_2,X_3,X_4,X_5)$. It can be easily checked that the numerator of $$g_2^{(1)}\left(X_0,X_1,X_2,X_3,\phi_1(X_0,X_1,X_2,X_3),\phi_2(X_0,X_1,X_2,X_3)\right)$$ 
is divisible by $g_2^{(0)}(X_0,X_1,X_2,X_3,X_4,X_5)$. 
This shows that $g_2^{(1)},\ldots,g_2^{(5)}$ are actually redundant, and hence $\mathcal{U}$ is defined by $g_1^{(0)},g_1^{(1)},g_2^{(0)}$. Since $g_1^{(0)}$ and $g_1^{(1)}$ are of degree one in $X_4$ and $X_5$ respectively, and $g_2^{(0)}\left(X_0,X_1,X_2,X_3,\phi_1(X_0,X_1,X_2,X_3),\phi_2(X_0,X_1,X_2,X_3)\right)$ does not vanish, $\mathcal{U}$ has dimension $3$. 
Now, $g_2^{(0)}(X_0,X_1,X_2,X_3,X_4,X_5)$ is an irreducible polynomial over $\mathbb{F}_4$ (this can be easily checked by MAGMA), and since it has coefficients in $\mathbb{F}_2$ and it is of degree 2 in all the variables, it is absolutely irreducible and so is $\mathcal{U}$.

As $\mathcal{V}$ is $\mathbb{F}_{q^6}$-projectively equivalent to $\mathcal{U}$, $\mathcal{V}$ is absolutely irreducible and of dimension $3$ as well, and by Corollary \ref{Th:AsyCafureMatera}, if $q$ is large enough there exist $q^3+O(q^{5/2})$ choices $(c_0,c_1,c_2,c_3,c_4,c_5)\in \mathrm{AG}(6,\F_q)\cap \mathcal{V}$ such that the corresponding 
$c=c_0\xi+c_1\xi^q+c_2\xi^{q^2}+c_3\xi^{q^3}+c_4\xi^{q^4}+c_5\xi^{q^5}$ satisfy $F_1(c)=F_2(c)=0$ and $F_3(c)\neq 0$ (as $\mathcal{U}$ is not contained in the hypersurface $X_0X_1X_2+1=0,$). Note that by Lemma \ref{Prop23} at most $q^2-q$ of such values of $c$ belong to $\F_{q^2}$. Moreover, there exists at most $(4q^2+q+2)+(2q^2+q+2)$ choices of $c \in \F_{q^6}$ such that $F_4(c)\cdot F_5(c)=0$ and so, for $q$ large enough $\mathfrak{C} \neq \emptyset$ and it has $q^3+O(q^{5/2})$ elements.
\end{proof}
\begin{remark}
\textnormal{As can be seen from the proof of Theorem \ref{scelteperc}, the degree of this  variety $\mathcal{V}$ is at most $6^3$. Therefore, if we apply Theorem \ref{Th:CafureMatera} to the number of points of $\mathcal{V}$ in $\mathrm{AG}(6,q)$, we obtain 
$$\#(\mathcal{V}\cap \mathrm{AG}(6,\F_q))\geq q^3- 215\cdot 214 q^{3-1/2}-5\cdot 216^{13/3} q^{2}.$$
It can be seen that the right-hand side in the above inequality is positive for $q>6\cdot 216^{13/3}$.}
\end{remark}

Now, suppose $q$ is sufficiently large, hence $\mathfrak{C} \not = \emptyset$ and let take some $c \in \mathfrak{C}$. By Lemma \ref{lemma-cons}, one shall study the possible solutions of System \eqref{system1}. Then we will investigate separately the cases depending on the roots of the polynomial  $r_3(X,Y,Z,T,U,V)$ in \eqref{r3} and hence those of $\bar{r}_3(X,Y,Z,T,V)$ in \eqref{fattorizzazione}.  Precisely, we will assume that $(x,y,z,t,u,v) \in \F^6_{q^6}$ is a solution of \eqref{system1} such that one of the following holds
\begin{itemize}
    \item[$1.$] $(x,y,z,t,u,v)=(0,0,0,0,0,0)$;
    \item[$2.$] $(x,y,z,t,u,v)\in (\F_{q^6}^*)^6 $ and $c^{q^5+q^2} + c^{q^2} + c^{q^5}yz + c^{q^5} + 1=0$;
    \item[$3.$] $(x,y,z,t,u,v)\in (\F_{q^6}^*)^6 $ and $\alpha x+ \alpha^q z=0$;
        \item[$4.$] $(x,y,z,t,u,v)\in (\F_{q^6}^*)^6 $ and $xy+\beta=0$;
    \item[$5.$] $(x,y,z,t,u,v)\in (\F_{q^6}^*)^6 $ and 
    $yz+\gamma=0$;

\end{itemize}
with $\alpha,\beta$ and $\gamma$ as in Proposition \ref{Prop:fattorizzazione}.\\

\vspace{0.3cm}
\noindent \textbf{Case 1.}\label{Case1}
Assume that $(x,y,z,t,u,v)=(0,0,0,0,0,0)$ is a solution of \eqref{system1}. Then $q_0(0,0,0,0,0,0)=0$ reads 
$$
(c^{q^3} + 1)(c^{q^4+q^2+1} + c + c^{q^2} + c^{q^4})=0.
$$
As $c \not = 1$, $c^{q^4+q^2+1} + c + c^{q^2} + c^{q^4}=0$, which, combined with \eqref{F1(c)} reads
\begin{eqnarray*}
&&B(c)=c^{2q^3+q^2+q+2}+c^{q^3+q^2+q+2}+c^{q^3+q^2+2}+c^{q^2+2}+c^{q^3+2q^2+q+1}+c^{q^3+q+1}+c^{q^3+2q^2+1}+\\&&c^{q^3+q^2+1}+c^{
q^2+1}+c+c^{2q^3+q^2+q}+c^{q^3+q^2+q}+c^{q^3}+1=0.
\end{eqnarray*}
By \eqref{prova3} and since $c \in \mathfrak{C}$, one gets that 
\begin{equation}\label{caso1cond}
\begin{split}
 0=&(c^2+1)F_2(c)+(c^{1+q^2} + 1)B(c)=c(
    c^{3+2q+q^2+q^3} + c^{3+q+q^2} + c^{2+2q+2q^2+q^3}+ c^{2+q+2q^2+q^3} + c^{2+q+q^2}\\
    &+ c^{2+q+q^3} + c^{2+2q^2+q^3}+ c^{2+2q^2} + 
    c^{2q^2+q^3} + c^2 + c^{1+2q+q^2+q^3}+ c^{1+q+3q^2+q^3}+ c^{1+q+q^2+q^3}+c^{1+q+q^2}\\
    &
    + c^{1+3q^2+q^3}+
    c^{1+2q^2+q^3}+ c^{1+2q^2}+ c^{1+q^2+q^3}+ c^{1+q^3}+ c + c^{2q+2q^2+q^3}+ c^{q+q^2}+ c^{2q^2+q^3}+ c^{q^2+q^3})
 \end{split}
 \end{equation}
and so,
\begin{equation}\label{cq3}
   (c+c^{q^2}) (c^{2+2q+q^2} + c^{1+q+2q^2}+ c^{1+q}+ c^{1+2q^2}+ c^{1+q^2}+ c^{2q+q^2}+ c^{q^2} + 1)c^{q^3}=(c + 1)
    (c^{2+q+q^2} + c^{1+2q^2}+ c + c^{q+q^2}).
\end{equation}
Now, if $c^{2+q+q^2} + c^{1+2q^2}+ c + c^{q+q^2}=0$, then, by \eqref{cq3}, yields
\begin{equation*}
    c^{2+2q+q^2} + c^{1+q+2q^2}+ c^{1+q}+ c^{1+2q^2}+ c^{1+q^2}+ c^{2q+q^2}+ c^{q^2} + 1=0.
\end{equation*}

Hence $(c,c^q,c^{q^2},c^{q^3},c^{q^4},c^{q^5})$ is the common root of the polynomials 
\begin{equation}
H_0(X_0,X_1,X_2,X_3,X_4,X_5)=X_0^2X_1X_2+X_0X_2^2+X_0+X_1X_2
 \end{equation}
 and 
 \begin{equation}
 \begin{split}
G_0(X_0,X_1,X_2,X_3,X_4,X_5)=&X_0{^2}X_1^2X_2+X_0X_1X_2^2+X_0X_1+X_0X_2^2\\
&+X_0X_2+X_1^2X_2+X_1^2+X_2+X_2+1
\end{split}
\end{equation}
and hence a  root of 
$\mathrm{Res}(H_0,G_0,X_2)$. This implies that $c$ satisfies  $c^{q+1}(c+1)^{q+4}=0$, a contradiction.
Therefore, $c^{2+q+q^2} + c^{1+2q^2}+ c + c^{q+q^2}\neq 0$ and from  \eqref{cq3}, we obtain

\begin{equation}
    c^{q^3}=\frac{(c + 1)
    (c^{2+q+q^2} + c^{1+2q^2}+ c + c^{q+q^2})}{(c + c^{q^2})
    (c^{2+2q+q^2} + c^{1+q+2q^2}+ c^{1+q}+ c^{1+2q^2}+ c^{1+q^2}+ c^{2q+q^2}+ c^{q^2} + 1)}.
\end{equation}
Putting this value in \eqref{prova3}, one has an expression whose numerator is
\begin{eqnarray*}
&&c^{q^2}
(c^{q^2} + 1)^{1+q+q^2}(c^{1+q^2} + 1)F_4(c)=0.
\end{eqnarray*}
Since $c \in \mathfrak{C}$, one obtains a contradiction.
\medskip

\vspace{0.6cm}
\noindent \textbf{Case 2.}\label{Case2} Assume that $(x,y,z,t,u,v)\in (\F_{q^6}^*)^6 $ is a solution of System \eqref{system1} such that $c^{q^5+q^2} + c^{q^2} + c^{q^5}yz + c^{q^5} + 1=0$. Then
\begin{equation*}
    z= \frac{(c+1)^{q^5+q^2}}{c^{q^5}y},\quad
    t=\frac{c^{q^5-1}y}{(c+1)^{q^5-q^3+q^2-1}},\quad
    u=\frac{(c+1)^{q^5+q^4-q^3+q^2+q-1}}{c^{q^5+q-1}y},\quad
    v=\frac{c^{q^5-q^2+q-1}y}{(c+1)^{q^4-q^3+q-1}},
\end{equation*}
and
$$
x=\frac{(c+1)^{q^4+q}}{c^{q^5+q^3-q^2+q-1}y}.
$$
Substituting these expressions in the equation $q_0(X,Y,Z,T,U,V)=0$, we get
$$
\frac{(c^{q^4+q^3+q+1}+ c^{q+1}+ c^{q^4+1} + c +c^{q^3+q}+ c^q)^{q^5+1}}{c^{1+q+q^5}(c^{q^3} + 1)}=0,
$$
and hence
$$
c^{q^4+q^3+q+1}+ c^{q+1}+ c^{q^4+1} + c +c^{q^3+q}+ c^q=0.
$$
By \eqref{ccq4_2}, this condition reads
$$
(c^{q^2+q+1}+c^{q+1}+c^{q^2+1}+1)c^{q^3}=c+1.
$$
Note that $c^{q^2+q+1}+c^{q+1}+c^{q^2+1}+1$ is not zero, otherwise $c=1$. 
Therefore,
\begin{equation}\label{ccq3case2}
c^{q^3}=\frac{c+1}{c^{q^2+q+1}+c^{q+1}+c^{q^2+1}+1},
\end{equation}
which combined with $F_2(c)=0$ yields
\begin{equation}\label{thirdfactor}
c^{q^2+q+1}(c+1)^{q+1}(c^{2q^2+q+2}+c^{q^2+q+2}+c^{2q^2+2}+c^{2q^2+q+1}+c^{q^2+q+1}+1)=0.
\end{equation}
Since $c \in \mathfrak{C}$ and the third factor in \eqref{thirdfactor} is $F_5(c)$, we get a contradiction.

\vspace{0.6cm}
\noindent \textbf{Case 3.}\label{Case (iii)} Assume that $(x,y,z,t,u,v)\in (\F_{q^6}^*)^6 $ is a solution of System \eqref{system1} such that $\alpha x + \alpha^q z=0$ with $\alpha=(c^{q^2+q+1}+1)^{q^3}$. Then,
\begin{equation*}
    z=\alpha^{1-q}x \quad
    t=\alpha^{q-q^2}y,\quad
    u=\frac{\alpha^{q^2+1}}{\alpha^{q^3+q}}x,\quad
    \textnormal{ and }
    v=\frac{\alpha^{q^3+q}}{\alpha^{q^4+q^2}}y.
\end{equation*}
Therefore,
$$
\frac{\alpha^{q^4+q^2+1}}{\alpha^{q^5+q^3+q}}x=x,
$$
and hence
$$\alpha^{q^4+q^2+1}+\alpha^{q^5+q^3+q}=0.$$
Therefore,
$$
(c^{q^2+q+1}+1)^{q^5+q^3+q}+(c^{q^2+q+1}+1)^{q^4+q^2+1}=0.
$$
After replacing $c^{q^4}$  and $c^{q^5}$ with the formulas obtained in \eqref{ccq4_2} and \eqref{ccq5}, respectively, one gets a rational function whose numerator is a polynomial $H_1(X)$, such that $c$ is its root. Then $(c,c^q,c^{q^2},c^{q^3},c^{q^4},c^{q^5})$ is a root of $\Phi(H_1)(X_0,X_1,X_2,X_3,X_4,X_5)$. Then the resultant $\mathrm{Res}(\Phi(H_1),F_2,X_3)$ evaluated in $(c,c^q,c^{q^2},c^{q^3},c^{q^4},c^{q^5})$ is
$$
c^{7q^2+10q+7}(c+1)^{5q^2+4q+4}(c^{q^2+q+1} + 1)^5(c^{2q^2+q+2} +c^{q^2+q+2} +c^{2q^2+2}  +c^{2q^2+q+1} +c^{q^2+q+1} + 1)^2=0.
$$
Therefore, as in Case 2, one gets a contradiction and System \eqref{system1} has no solution.

\vspace{0.6cm}
\noindent \textbf{Case 4.}\label{Case (iv)} Assume that $(x,y,z,t,u,v)\in (\F_{q^6}^*)^6 $ is a solution of System \eqref{system1} such that $xy+\beta=0$. Then, substituting in $q_1(X,Y,Z,T,U,V)$ the expressions

\begin{equation*}
    y=\frac{\beta}{x}\quad
    z=\beta^{q-1}x,\quad
    t=\frac{\beta^{q^2+1-q}}{x},\quad
    u=\beta^{q^3-q^2+q-1}x,\quad
    v=\frac{\beta^{q^4-q^3+q^2-q+1}}{x},
\end{equation*}
and the expressions of $\beta$, $c^{q^4}$ and $c^{q^5}$ as in \eqref{beta}, \eqref{ccq4_2} and \eqref{ccq5} respectively, we get a rational  function whose numerator is $H_2(X)$ and $H_2(c)=0$. A MAGMA aided computation shows that the resultant of $\Phi(H_2)(X_0,X_1,X_2,X_3,X_4,X_5)$ and $F_2(X_0,X_1,X_2,X_3,X_4,X_5)$ with respect to $X_3$ and evaluated in $(c,c^q,c^{q^2},c^{q^3},c^{q^4},c^{q^5})$ is

$$
c^{6q^2+6q+5}(c+1)^{4q^2+5q+4}(c^{q^2+1}+1)(c^{q^2+q+1}+1)^5(c^{2q^2+q+2} +c^{q^2+q+2} +c^{2q^2+2}  +c^{2q^2+q+1} +c^{q^2+q+1} + 1)^3=0.
$$

Therefore, a contradiction follows again.

\vspace{0.6cm}
\noindent \textbf{Case 5.}\label{Case (v)}  Finally, assume that $(x,y,z,t,u,v)\in (\F_{q^6}^*)^6 $ is a solution of System \eqref{system1} such that $yz+\gamma=0$, where $\gamma$ is as in \eqref{gamma}. 
Then, substituting  in $q_1(X,Y,Z,T,U,V)$  the expressions
\begin{equation*}
    z=\frac{\gamma}{y},\quad
    t=\gamma^{q-1} y,\quad
    u=\frac{\gamma^{q^2-q+1}}{y},\quad
    v=\gamma^{q^3-q^2+q-1}y,\quad
    x=\frac{\gamma^{q^4-q^3+q^2-q+1}}{y},
\end{equation*}
and the expressions of $\gamma$, $c^{q^4}$ and $c^{q^5}$ as in \eqref{gamma}, \eqref{ccq4_2} and \eqref{ccq5} respectively, we get a rational function whose numerator is $H_3(X)$.
Then, MAGMA aided computation shows that the resultant of $\Phi(H_3)(X_0,X_1,X_2,X_3,X_4,X_5)$ and $F_2(X_0,X_1,X_2,X_3,X_4,X_5)$ with respect to $X_3$ and evaluated in $(c,c^q,c^{q^2},c^{q^3},c^{q^4},c^{q^5})$ is
$$
c^{11q^2+14q+12}(c+1)^{9q^2+9q+8}(c^{q^2+1}+1)(c^{q^2+q+1}+1)^6(c^{2q^2+q+2} +c^{q^2+q+2} +c^{2q^2+2}  +c^{2q^2+q+1} +c^{q^2+q+1} + 1)^3=0.
$$
Again, a contradiction because of $c \in \mathfrak{C}$.

Hence, taking Lemma \ref{prop:equiv} into account, we can state the following result.
\begin{theorem}\label{maintheorem}
   Let 
   $$U_c=\{(x,x^q+x^{q^3}+cx^{q^5}): x \in \F_{q^6}\} \subseteq \F_{q^6} \times \F_{q^6}, \quad c \in \F_{q^6},$$ 
  and let 
   $$\mathfrak{C}=\{c \in \F_{q^6} \setminus \F_{q^2} : F_1(c)=F_2(c)=0 \textnormal{ and } F_3(c)F_4(c)F_5(c)\neq 0\}$$ (cf. \eqref{ccq4}, \eqref{prova3}, \eqref{prova31}, \eqref{F4} and \eqref{F5}). For $q$ large enough, $\mathfrak{C}$ has size $q^3+O(q^{5/2})$ and, for any $c\in \mathfrak{C}$, $U_c$ is maximum scattered. Moreover,  $U_{c}$ is not $\Gamma \mathrm{L}(2,q^6)$-equivalent to the previously known maximum scattered subspaces.  \end{theorem}

   \section{The RD-code $\mathcal{D}_{c,s}$} \label{MRDcode}
Let consider the $\F_{q^n}$-linear subspace of $\tilde{\mathcal{L}}_{n,q,\sigma}$
\begin{equation}
\cC_f = \{aX + bf(X) : a, b \in \F_{q^n} \} = \langle X, f(X) \rangle_{\F_{q^n}}.
\end{equation}
with $f \in \tilde{\mathcal{L}}_{n,q,\sigma}$. This defines an RD code and, as recalled in Section \ref{Section1}, $\cC_f$ is an MRD code with mininum distance $n-1$ if and only if $f(X)$ is a scattered polynomial. In this section we will give some results about the code 
\begin{equation}\label{codeD}
\mathcal{D}_{c,s}=\mathcal{D}_{f_{c,s}}=\langle X, f_{c,s}(X) \rangle_{\F_{q^6}} \leq \tilde{\mathcal{L}}_{6,q,s} 
\end{equation}
with $f_{c,s}(X)=X^{q^s}+X^{q^{3s}}+cX^{q^{5s}}$, for $q$ even sufficiently large, $c \in \mathfrak{C}$ and $s \in \{1,5\}$. 
Firstly, we stress that in the previous section we showed that $f_{c,1}(X)=x^q+x^{q^{3}}+cx^{q^{5}}$ is scattered, but applying the same argument as \cite[Section 3]{NSZ} one may state that $f_{c,5}(X)=x^{q^5}+x^{q^{3}}+cx^{q}$ is scattered as well, and since $U_{c,5}$ is $\Gamma \mathrm{L}$-equivalent to 
$U_{1/c,3}$ (cf. \eqref{Ubc}), by Lemma \ref{prop:equiv}, hence $U_{c,5}$ is not equivalent to known scattered $\F_q$-subspaces of $\F_{q^6}^2$ for $c \not =1$.\\
In order to study the code $\mathcal{D}_{c,s}$, we shall recall some definitions and results 
which, although defined and proved for $q$-linearized polynomials, hold for $\sigma$-linearized polynomials as well.\\
Let $f(X)$ be a scattered $\sigma$-linearized polynomial and consider $G_f$, the stabilizer in $\mathrm{GL}(2,q^n)$
 of the
scattered subspace $U_f$. By \cite[Lemma 4.1]{LMTZ} (see also \cite[Remark 3.4]{LZ-standard}), $G^\circ_f = G_f \cup \{O\}$, where $O$ is the null matrix of order 2, is a subfield of matrices isomorphic to $I_R(\cC_f)$ as a field and in  \cite{somi}
the authors showed the following result.
\begin{theorem}\cite[Theorem 2.1]{somi}
Let ${\cal C}_f=\langle X,f(X) \rangle_{\F_{q^n}} \leq \tilde{\mathcal{L}}_{n,q}$ be a linear $2$-dimensional MRD code with minimum distance $n-1$ where $f(X)=\sum^{n-1}_{i=1}a_i X^{q^i}$.
The following statements are equivalent:
\begin{enumerate}
\item [$(i)$] $|G^\circ_f|=q^m$, $m > 1$, and all elements of $G^\circ_f$ are diagonal.
\item [$(ii)$] The $q$-linearized polynomial $f(X)$ is of the shape
\begin{equation}\label{standard}
\sum^{n/m -1}_{i=1}b_iX^{q^{mi+r}}
\end{equation}
where $m=\gcd(\{(i - j) \pmod n : b_ib_j \neq 0 \textnormal{ with }  i \neq j\} \cup \{n\})> 1$
and $\gcd(r,m)=1$, $1 \leq r  \leq m-1$. 
\item [$(iii)$] For the right idealizer it holds that $$I_R(\cC_f)=\{\alpha X \,:\, \alpha \in \F_{q^{m}}\}.$$
\end{enumerate}
Moreover, if one of the above conditions holds, then
\begin{equation}
    G^\circ_f=\left \{\begin{pmatrix}
    \alpha & 0 \\
     0 & \alpha^{q^r}   \end{pmatrix} \colon \alpha \in \F_{q^m} \right \}.
\end{equation}

\end{theorem}
The scattered polynomials of the form as in \eqref{standard} are studied in \cite{LZ-standard} where they are called polynomials in \textit{standard form}. Some properties of the code $\cC_f$ with $f(X)$ in this form are collected in \cite{somi}.
Therefore, since the scattered polynomial $f_{c,s}(X)$ verifies condition $(ii)$ of the theorem above  as a simple byproduct one has the following result.

\begin{theorem}
    Let $f(X)=X^{q^s}+X^{q^{3s}}+cX^{q^{5s}}\in \tilde{\mathcal{L}}_{6,q,s}$ be the scattered polynomial with $c \in \mathfrak{C}$, $s \in \{1,5\}$, and consider  the MRD-code $\mathcal{D}_{c,s}=\langle X,f_{c,s}(X) \rangle_{\F_{q^6}}$. Then,
    \begin{itemize}
        \item [$i)$]  the linear automorphism group of $U_{c,s}$ is
 \begin{equation}
    G_{c}=\left \{\begin{pmatrix}
    \alpha & 0 \\
     0 & \alpha^{q}   \end{pmatrix} \colon \alpha \in \F^*_{q^2} \right \},
\end{equation}        
    \item [$ii)$] the right idealizer is 
    $$I_R(\cD_{c,s})=\{\alpha x : \alpha \in \F_{q^2}\}.$$
    \end{itemize}
\end{theorem}
In \cite[Theorem 4.5]{LZ-standard}, it is shown that a scattered polynomial in standard form is bijective. Then, we may state the following.
\begin{lemma}
If $f_{c,s}(X)=X^{q^s}+X^{q^{3s}}+cX^{q^{5s}} \in \tilde{\mathcal{L}}_{6,q,s}$, $s \in \{1,5\}$, is scattered, then $\mathrm{N}_{q^6/q^2}(c) + \Tr_{q^6/q^2}(c) \neq 0$.
\end{lemma}
\begin{proof}
 The result follows easily noting that $f_{c,s}(X)=g_{c,s}(X^{q^s})$ with $g_{c,s}(X)=X+X^{q^{2s}}+cX^{q^{4s}}$ and it is bijective if and only if $g_{c,s}(X)$ is. Then, this is equivalent to the fact that the following matrix
 \begin{equation*}
 \begin{pmatrix}
     1 & 1 & c \\
     c^{q{2s}} & 1 & 1 \\
     1 & c^{q^{4s}} & 1
\end{pmatrix}
 \end{equation*}
has full rank.
\end{proof}

Note that for $c \in \mathfrak{C}$, the quantity $\mathrm{N}_{q^6/q^2}(c) + \Tr_{q^6/q^2}(c)$ cannot be equal to zero as  already pointed out in Case 1 of Section
\ref{presentazione}.
Now, for $q$ large enough, consider the class of codes 
$$\mathfrak{D}=\{\mathcal{D}_{c,s} : c \in \mathfrak{C}, s \in \{1,5\} \}.$$ In the remainder of this section we will provide the sufficient and necessary conditions for two codes  in $\mathfrak{D}$ to be equivalent.
In order to do that, we recall the following result about the equivalence of two linear MRD codes, $2$-dimesional over $\F_{q^n}$.

\begin{theorem}\label{diagonal} \cite[Theorem 2.2]{somi}
Let ${\cal D}_{f_1}$, ${\cal D}_{f_2}$ be two $2$-dimensional MRD codes where $f_1, f_2$ are scattered polynomials in standard form. Then, they are equivalent if and only if there exist $A,B,C,D \in \F_{q^n}^*$ such that
\begin{equation}\label{eq-standard}
    Df_2(X)=f_1^\rho(AX) \quad \textnormal{or} \quad f_1^\rho(Bf_2(X))=CX,
\end{equation}
for some $\rho \in \mathrm{Aut}(\F_{q^n})$.
In particular, 
\begin{itemize}
    \item [$(i)$] $\cD_{f_2}=\cD^\rho_{f_1} \circ AX$ or,
    \item [ $(ii)$] $\cD_{f_2}=\cD^\rho_{f_1} \circ B f_2(X)$.
\end{itemize}
\end{theorem}

As an application of this result, we can deal with the equivalence problem in the class $\mathfrak{D}$  of MRD codes. Before that, 
a preliminary lemma is needed.
\begin{lemma} \label{nonorm1}
Let $c \in \mathfrak{C}$. Then 
\begin{itemize}
    \item [$(i)$] for any $\rho \in \mathrm{Aut}(\F_{q^6}), c^\rho \in \mathfrak{C}$;
    \item [$(ii)$] $\mathrm{N}_{q^6/q^2}(c) \not = 1$.
\end{itemize}
\end{lemma}
\begin{proof}
Because of the  polynomials defining the set $\mathfrak{C}$, $(i)$ is trivial. $(ii)$ Suppose that $c \in \mathfrak{C}$ and $\mathrm{N}_{q^6/q^2}(c) = 1$.
    Then $(c,c^q,c^{q^2},c^{q^3},c^{q^4},c^{q^5})$ is a root of the polynomial 
    $$N(X_0,X_1,X_2,X_3,X_4,X_5)=X_0X_2X_4+1.$$
Let $F_1(X)$ be the polynomial as in \eqref{ccq4}  and 
 consider $\Phi(F_1)(X_0,X_1,X_2,X_3,X_4,X_5)$. By a MAGMA aided computation,
one has that 
$$\mathrm{Res}(\Phi(F_1),N,X_4)=X_0X_2(X_0X_1X_3+X_0+X_1X_2X_3^2+X_1X_2X_3+X_1X_3^2+1).$$
Calculating the resultant between the polynomial just above and $\Phi(F_2)(X_0,X_1,X_2,X_3,X_4,X_5)$ with respect to $X_3$ (cf. \eqref{prova3}), one gets
\begin{equation}
    X_0^2X_1X_2^2(X_2 + 1)(X_1+1)(X_0+1)(X_0X_1X_2+1)(X_0^2X_1X_2^2+X_0^2X_1X_2+X_0^2X_2^2+X_0X_1X_2^2+X_0X_1X_2+1),
\end{equation}
where the last factor is $\Phi(F_5)$. Since $c \in \mathfrak{C}$, $(c,c^q,c^{q^2},c^{q^3},c^{q^4},c^{q^5})$ cannot be a root of this polynomial, a contradiction follows.
\end{proof}

So, we can now prove the following.
\begin{theorem}\label{equivalence}
Let $\cD_{c_1,s}$ and $\cD_{c_2,t}$ be MRD-codes with $c_1,c_2 \in \mathfrak{C}$ and $s,t  \in \{1,5\}$. Then, they are equivalent 
if and only if there exists $\rho \in \mathrm{Aut}(\F_{q^6})$ such that one of the following holds:
\begin{itemize}
    \item [$i)$] $s=t$ and $c_1^\rho=c_2$;
    \item [$ii)$] $s+t=6$ and $c_1^\rho=\frac{(c_2^{q^{3s}+q^s}+1)(c_2^{q^{5s}}+1)}{(c_2+1)^{q^s+q^{3s}}}$.
\end{itemize}

\end{theorem}
\begin{proof}
Let $f_{1,s}(X):=f_{c_1,s}(X)$ and $f_{2,t}(X):=f_{c_2,t}(X)$ with $c_1,c_2 \in \mathfrak{C}$. 
By Theorem \ref{diagonal}, the codes $\cD_{c_1,s}$ and $\cD_{c_2,t}$ are equivalent if and only if there exist $A,B,C,D \in \F^*_{q^6}$ and $\rho \in \mathrm{Aut}(\F_{q^6})$ such that 
\begin{equation}\label{parigi}
    Df_{2,t}(X)=f_{1,s}^\rho(AX) \quad \textnormal{or} \quad f_{1,s}^\rho(Bf_{2,t}(X))=CX.
\end{equation}
Note that $s \equiv \pm t \pmod 6$  are the only cases to analyze. First suppose that $s \equiv t \pmod 6$. 
If the first condition in \eqref{parigi} holds, we get 
\begin{equation*}
\begin{cases}
    D=A^{q^s}\\
    D=A^{q^{3s}}\\
    Dc_2=c^\rho_1 A^{q^{5s}}
    \end{cases}
\end{equation*}
whence $D=A^{q^s}$, $A\in \F_{q^2}$, and $c_1^{\rho}=c_2$. 
If the second condition in \eqref{parigi} holds, then considering the coefficients of $X,X^{q^{2s}},X^{q^{4s}}$, respectively, one gets the following linear system 
\begin{equation} \label{linearsystem}
\begin{cases}
    c^{q^s}_2B^{q^s}+B^{q^{3s}}+c_1^\rho B^{q^{5s}}=C\\
    B^{q^s}+c_2^{q^{3s}}B^{q^{3s}}+c^\rho_1B^{q^{5s}}=0\\
    B^{q^s}+B^{q^{3s}}+c^\rho_1c^{q^{5s}}_2B^{q^{5s}}=0
\end{cases}
\end{equation}
in the unknowns $B^{q^s},B^{q^{3s}},B^{q^{5s}}$. Let $f_{2,s}(X)$ defined as the beginning of the proof with $t=s$. Since $f_{2,s}(X)$ is bijective, the matrix
\begin{equation*}
\begin{pmatrix}
      c^{q^s}_2 & 1 & c_1^\rho\\
    1& c_2^{q^{3s}} & c^\rho_1 \\
    1 & 1 & c^\rho_1c^{q^{5s}}_2 
\end{pmatrix}
\end{equation*}
has determinant 
\begin{equation}\label{determinant}
    c^\rho_1(\mathrm{N}_{q^6/q^2}(c_2)+\mathrm{Tr}_{q^6/q^2}(c_2))^{q^s} \neq 0.
\end{equation}
Then, 
\begin{equation}\label{solsystem}
\begin{split}
    &B^{q^s}=\frac{(c_2^{q^{3s} + q^{5s}}+1)C}{(\mathrm{N}_{q^6/q^2}(c_2)+\mathrm{Tr}_{q^6/q^2}(c_2))^{q^s}},\\
    &B^{q^{3s}}=\frac{(c_2^{q^{5s}}+1) C}{(\mathrm{N}_{q^6/q^2}(c_2)+\mathrm{Tr}_{q^6/q^2}(c_2))^{q^s}},\\
    &B^{q^{5s}}=\frac{(c_2^{q^{3s}}+1) C}{c^\rho_1(\mathrm{N}_{q^6/q^2}(c_2)+\mathrm{Tr}_{q^6/q^2}(c_2))^{q^s}}.
\end{split}
\end{equation}
Computing the ratios between the expressions in \eqref{solsystem}, we have
\begin{equation*}
    B^{q^{3s}-q^s}=\frac{c_2^{q^{5s}}+1}{c_2^{q^{3s}+q^{5s}}+1},\quad \quad B^{q^{5s}-q^{3s}}= \frac{c_2^{q^{3s}}+1}{c^\rho_1(c_2^{q^{5s}}+1)} \textnormal{\quad and \quad} B^{q^s-q^{5s}}=\frac{c^{\rho}_1(c_2^{q^{3s}+q^{5s}}+1)}{c_2^{q^{3s}}+1}.
\end{equation*}
Raising the first expression to the $q^{2s}$-th power and to the $q^{4s}$-th power and equating with the second and the third one, respectively, we get
\begin{equation*}
    c_1^\rho=\frac{(c_2^{q^{3s}}+1)(c_2^{q^{5s}+q^s}+1)}{(c_2+1)^{q^s+q^{5s}}}=\frac{(c_2+1)^{2q^{3s}}}{(c_2^{q^s+q^{3s}}+1)(c_2^{q^{3s}+q^{5s}}+1)}
\end{equation*}
and hence 
\begin{equation*}
    (c_2+1)^{q^s+q^{3s}+q^{5s}}=(c_2^{q^s+q^{3s}}+1)(c_2^{q^{3s}+q^{5s}}+1)(c_2^{q^s+q^{5s}}+1).
\end{equation*}
This is equivalent to 
\begin{equation}
   (c_2^{q^s+q^{3s}+q^{5s}} + 1)(\mathrm{N}_{q^6/q^2}(c_2)+\mathrm{Tr}_{q^6/q^2}(c_2))^{q^s}=0, 
\end{equation}
a contradiction by Lemma \ref{nonorm1} and  \eqref{determinant} and this concludes the first part of the proof.\\
Assume now $t\equiv -s \pmod{6}$. 
As before, $\cD_{c_1,s}$ is  equivalent to $\mathcal{D}_{c_2,-s}$ if and only if there exist $A,B,C,D \in \F_{q^6}$ and $\rho \in \mathrm{Aut}(\F_{q^6})$ such that 
\begin{equation}
    Df_{2,-s}(X)=f^\rho_{1,s}(AX) \,\, \textnormal{ or }\,\,f^\rho_{1,s}(Bf_{2,-s}(X))=CX.
\end{equation}
If the first case occurs, one gets the following conditions
\begin{equation}
\begin{cases}
    Dc_2=A^{q^s}\\
    D=A^{q^{3s}} \\
    D=c_1^\rho A^{q^{5s}}
    \end{cases},
\end{equation}
$c_2=A^{q^s-q^{3s}}$ and $c^\rho_1=A^{q^{3s}-q^{5s}}$, whence $\mathrm{N}_{q^6/q^2}(c_2)=\mathrm{N}_{q^6/q^2}(c_1^\rho)=1$. By Lemm \ref{nonorm1}, this is a contradiction.
Then, $f^\rho_{1,s}(Bf_{2,-s}(X))=CX$. From   this equality, one gets the following linear system in the unknowns $B^{q^s},B^{q^{3s}},B^{q^{5s}}$
\begin{equation}
\begin{cases}
    B^{q^s}+B^{q^{3s}}+c_1^\rho c_2^{q^{5s}}B^{q^{5s}}=C\\
    B^{q^s}c_2^{q^s}+B^{q^{3s}}+c_1^\rho B^{q^{5s}}=0\\
    B^{q^s}+c_2^{q^{3s}}B^{q^{3s}}+c_1^\rho B^{q^{5s}}=0.
\end{cases}
\end{equation}
This system, arguing as for \eqref{linearsystem}, has the unique solution
\begin{equation}
\begin{split}
    &B^{q^s}=\frac{(c_2^{q^{3s}}+1)C}{(\mathrm{N}_{q^6/q^2}(c_2)+\mathrm{Tr}_{q^6/q^2}(c_2))^{q^s}},\\
    &B^{q^{3s}}=\frac{(c_2^{q^s}+1) C}{(\mathrm{N}_{q^6/q^2}(c_2)+\mathrm{Tr}_{q^6/q^2}(c_2))^{q^s}},\\
    &B^{q^{5s}}=\frac{(c_2^{q^s+q^{3s}}+1) C}{c^\rho_1(\mathrm{N}_{q^6/q^2}(c_2)+\mathrm{Tr}_{q^6/q^2}(c_2))^{q^s}}.
\end{split}
\end{equation}
Using the same argument of the previous case, we get that 
$$c_1^\rho=\frac{(c_2^{q^{3s}+q^s}+1)(c_2^{q^{5s}}+1)}{(c_2+1)^{q^s+q^{3s}}}$$ must hold, whence
$$B=\frac{(c_2+1)^{q^{2s}} \mu^{q^s}}{(\mathrm{N}_{q^6/q^2}(c_2)+\mathrm{Tr}_{q^6/q^2}(c_2))(c_2+1)^{1+q^{2s}}}\,\,\textnormal{ and } \,\,C=\frac{\mu}{(c_2+1)^{q^s+q^{3s}}},$$ with  $\mu \in \F_{q^2}$. This concludes the proof. 
\end{proof}

Although, it seems not easy to prove that if $c \in \mathfrak{C}$ then $\frac{ (c^{q^{3s}+q^s}+1  )(c^{q^{5s}}+1)}{(c+1)^{q^s+q^{3s}}} \in \mathfrak{C}$ with $s \in \{1,5\}$ as well (cf. Theorem \ref{equivalence}), we can determine a lower bound on the number of equivalence classes of the codes in $\mathfrak{D}$.

\begin{cor}
     If $q=2^e$ is large enough, then the number of equivalence  classes of the codes $\cD_{c,s}$ is at least $\frac{|\mathfrak{C}|}{6e}$.
\end{cor}
\begin{proof}
By Theorem \ref{equivalence}, any element of $\cD_{c,s}$ is equivalent to at most $12e$ other elements and the claim follows.
    
\end{proof}

\section{The $\F_q$-linear set arising from $U_{c,s}$}\label{linearset}

Let $V=\F_{q^n} \times \F_{q^n}$ and let $\Lambda=\PG(V,q^n)=\PG(1,q^n)$ be the projective line over $V$. Consider an $\F_q$-linear $m$-dimensional subspace $U$  of $V$, the set of points  
$$L = L_U = \{\langle \textbf{u} \rangle_{\F_{q^n}} \colon \textbf{u} \in  U \setminus \{\textbf{0}\}\} \subset \Lambda $$ 
is called a \textit{linear set} of rank $m$.
It is straightforward to see that a linear set $L_U$ of rank $m$ may have size at most $(q^{m}-1)/(q-1)$. This size is reached when the underlying vector space $U$ is scattered and $L$ is said to be \textit{scattered} as well. Moreover, $L$ is called \textit{maximum scattered} if the underlying $\F_q$-linear subspace $U$ is scattered and has rank $n$.\\
Two linear sets $L_U$ and $L_V$ of $\PG(1, q^n)$ are said to be $\mathrm{P\Gamma L}$-\textit{equivalent} (or simply \textit{equivalent}) if there
exists $\phi \in \mathrm{P\Gamma L}(2, q^n)$ mapping $L_U$ to $L_V$. It is clear that if $U$ and $V$ are on the same $\mathrm{\Gamma L}$-orbit then $L_U$ and $L_V$ are equivalent, but the converse does not hold, see \cite{CMP,CSZ2015}. So in \cite[Definitions 2.4 and 2.5]{CMP}, the notions of $\mathcal{Z}(\mathrm{\Gamma L})$-class and $\mathrm{\Gamma L}$-class of a linear set arise naturally: an $\F_q$-linear set $L_U$ has $\mathcal{Z}(\mathrm{\Gamma L})$-\textit{class} $r$ if $r$ is the largest integer
such that there exist $\F_q$-subspaces $U_1,U_2,\ldots, U_r$ of $V$ with
\begin{itemize}
    \item [$a)$] $L_{U_i} = L_U$ for
$i \in \{1, 2,\ldots,r\}$,
\item [$b)$] $U_i \neq \lambda U_j$ for any $\lambda \in  \F^*_{q^n}$ and distinct $i, j \in \{1, 2,\ldots,r\}$.
\end{itemize} 
The linear set  $L_U$ has $\mathrm{\Gamma L}$-\textit{class} $s$ if $s$ is the largest integer such that there exist $\F_q$-subspaces $U_1, U_2,\ldots, U_s$ of $V$ with \begin{itemize} 
\item [$a)$]  $L_{U_i} = L_U$ for $i \in \{1, 2,\ldots,s\}$,
\item [$b)$] $U_i$ and $U_j$ are not on the same $\mathrm{\Gamma L}(2, q^n)$-orbit for distinct $i, j \in \{1, 2,\ldots,s\}$.
\end{itemize}
If $s = 1$, then $L_U$ is called \textit{simple} and for each $\F_q$-subspace
$W$ of $V$, $L_U = L_W$ only if $U$ and $W$ are in the same orbit of $\mathrm{\Gamma L}(2, q^n)$.
Now, to make the notation easier, we will denote by $L^{1,n}_{s}$, $L^{i,n}_{s,\delta}$, $i \in \{2,3\}$ and by $L_{c,s}$ the $\F_q$-linear sets whose underlying scattered vector space
is $U^{1,n}_s$, $U^{i,n}_{s,\delta}$, $ i \in \{2,3\}$ and $U_{c,s}$, respectively. 
Since $q$ is even, in order to state that the $\F_q$-linear set $L_{c,s}$ is new, we shall compare it with the other ones existing in even characteristic. The $\F_q$-linear sets $\mathrm{P \Gamma L}(2, q^n)$-equivalent to $L^{1,n}_s$ and $L^{2,n}_s$ are
called of \textit{pseudoregulus} and \textit{LP type}, respectively. It is very easy to show that $L^{1,n}_1=L^{1,n}_s$ for any $s$ coprime to $n$.
In \cite[Theorem 3]{LP2001}, the authors proved that $L^{2,n}_{1,\delta}$ and $L^{1,n}_1$ are not
$\mathrm{P\Gamma L}(2, q^n)$-equivalent when $q > 3$, $n  \geq  4$ and 
 by \cite[Theorem 4.4]{CsMZ2018}, the linear sets $L^{1,n}_{1}, L^{i,n}_{s,\delta}$, $i\in \{2,3\}$ are pairwise non-equivalent for any choice of $s$ and $\delta$ for $n \in \{5,6,8\}$. 
These results are based on the fact that for  a linear set of LP type is of  $\mathrm{\Gamma L}$-class at most $2$, i.e.  a linear set $L_U$ is equivalent
to $L^{2,n}_{s,\delta}$ if and only if $U$ is $\mathrm{\Gamma L}(2, q^n)$ equivalent to either $U^{2,n}_{s,\delta}$ or to $U^{2,n}_{n-s,\delta^{-q^{n-s}}}$, and $L^{3,n}
_{s,\delta}$
is a simple linear set for any choice of definining parameters,
  \cite[Proposition 4.1 and 4.2]{CsMZ2018}. 

So, as byproduct of Lemma \ref{prop:equiv} and the previous section, we can state the following.
\begin{theorem}
Let $L_{c,s}$  be  the maximum scattered $\F_q$-linear set of $\PG(1, q^6)$ defined by the $\F_q$-subspace 
$$U_{c,s} = \{(x, x^{q^s} + x^{q^{3s}}
+ cx^{q^{5s}})\colon  x \in \F_{q^6} \}\subseteq \F_{q^6} \times \F_{q^6},$$
with $q$ even large enough, $c \in \mathfrak{C}$ and $s \in \{1,5\}$. Then, $L_{c,s}$ is
not $\mathrm{P \Gamma L}(2, q^6)$-equivalent to the previously known maximum scattered $\F_q$-linear sets of
$\PG(1, q^6)$.
\end{theorem}

\section*{Acknowledgements}
The authors thank the Italian National Group for Algebraic and Geometric Structures and their Applications (GNSAGA—INdAM)
which supported the research. The fourth author is funded by the project ``Metodi matematici per la firma digitale ed il cloud computing" (Programma Operativo Nazionale (PON) ``Ricerca e Innovazione" 2014-2020, University of Perugia) and by the INdAM-GNSAGA project CUP\textunderscore E55F22000270001 ``Curve algebriche e loro applicazioni".

\section*{Declarations}
{\bf Conflicts of interest.} The authors have no conflicts of interest to declare that are relevant to the content of this
article.

\bigskip

\noindent Daniele Bartoli, Marco Timpanella\\
Department of Mathematics and Computer Sciences\\ University of Perugia,\\ Via Vanvitelli 1-06123 Perugia,  Italy\\
 \texttt{\{daniele.bartoli, marco.timpanella\}@unipg.it}
\bigskip

\noindent Giovanni Longobardi, Giuseppe Marino\\
Department of Mathematics and its Applications ``Renato Caccioppoli",\\
University of Naples Federico II,\\
Via Cintia, Monte S.Angelo I-80126 Napoli, Italy\\
\texttt{\{giovanni.longobardi, giuseppe.marino\}@unina.it}

\end{document}